\documentclass{amsart}
\usepackage{amsfonts,amssymb,amsmath,amsthm}
\usepackage{url}
\usepackage{enumerate}
\usepackage[dvips]{graphicx} 
\usepackage{cmdtrack} 
\usepackage{mathrsfs}

\newtheorem{theorem}{Theorem}[section]
 \newtheorem{prop}{Proposition}[section]

\theoremstyle{definition}

 \newtheorem{exam}{Example}[section]
 
 \newtheorem{dfn}{Definition}[section]
\theoremstyle{remark}

 \numberwithin{equation}{section}

\author{Bishnu Hari Subedi}
\address{
Central Department of Mathematics\\
Institute of Science and Technology\\
Tribhuvan University\\
Kirtipur, Kathmandu, Nepal}
\email{subedi.abs@gmail.com}
\thanks{This research work is supported by a  Ph.D Faculty Fellowship of the University Grants Commission, Nepal and NMS-Nick Simons Fellowship of the Nepal Mathematical Society, Nepal}

\begin{document}

\title[A study of Holomorphic Semigroups]{A study of Holomorphic Semigroups}

\begin{abstract}
In this paper,  we investigate some characteristic features of holomorphic semigroups. In particular, we investigate nice examples of holomorphic semigroups whose every left or right ideal includes minimal ideal. These examples are compact topological holomorphic semigroups and examples of compact topological holomorphic semigroups are the spaces of ultrafilters of semigroups. 
\end{abstract}
\maketitle

\section{Introduction} 

Semigroups are very classical algebraic structures with a binary composition that satisfies associative law. It naturally arose from the general mapping of a set into itself. So, a set of holomorphic functions on the complex plane $ \mathbb{C} $ or Riemann sphere $ \mathbb{C}_{\infty} $ or certain subsets thereof naturally forms a semigroup.
We call that any analytic map on $ \mathbb{C} $ or $ \mathbb{C_{\infty}} $ or on certain subsets thereof by a holomorphic function. 
 
\begin{dfn}[\textbf{Holomorphic semigroup and subsemigroup}]
A \emph{holomorphic semigroup} $ S $ is a semigroup of holomorphic functions defined on  $ \mathbb{C} $ or  $\mathbb{C}_{\infty} $ or certain subsets thereof with the semigroup operation being functional composition. A non-empty subset $ T\subseteq S $ is a \emph{subsemigroup} of $ S $ if $ f \circ g \in T $ for all $ f, g \in S $. 
\end{dfn}  
Let 
\begin{equation}\label{hs}
\mathscr{F} =\{f_{\alpha}: f_{\alpha} \; \text{is a holomorphic function  for all}\; \alpha \in \Delta\}, 
\end{equation}
where index set $ \Delta $  is allowed to be infinite in the case of discrete semigroups and uncountable in the case of continuous semigroups.  In this sense, a holomorphic semigroup is a set $ S$ of holomorphic functions from $\mathscr{F}$ such that $ f_{\alpha+ \beta}(z) = f_{\alpha}(f_{\beta}(z)) $ for all $ z \in \mathbb{C} $ or  $ \mathbb{C_{\infty}} $ or certain subsets thereof and for all $ \alpha, \beta, \alpha+\beta \in \Delta $.  If $ \Delta \subseteq \mathbb{N} $ then $ S $ is called \emph{one parameter discrete semigroup} and such semigroup consists all iterates $f^{n}, n \in \mathbb{N}  $ of the function $ f $. If $ \Delta = (0, T), T >0 $, an interval of the real line $ \mathbb{R} $, then $ S $ is called \emph{one parameter continuous semigroup} and the iterates $ f^{t} $ with $ t \in (0, T), T > 0 $ is a fractional iterates of $ f $. 

We are more interested in special holomorphic semigroups  whose each element can be expressed as a finite composition of certain  holomorphic functions. More formally, such a semigroup is defined as follows:
\begin{dfn} [\textbf{Holomorphic semigroup generated by holomorphic functions}]
Let $\mathscr{F}$ be a family of holomorphic functions as defined in \ref{hs}.    A holomorphic semigroup $ S $ generated by  $\mathscr{F} $ is a semigroup  of all elements that can be expressed as a finite composition of elements in $ \mathscr{F}$. We denote such holomorphic semigroup by 
$$ 
S  =\langle f_{\alpha} \rangle _{\alpha \in \Delta}\;   \text{or simply}\; S =\langle f_{\alpha} \rangle. 
$$ 
\end{dfn}
Holomorphic semigroup $ S $ is said to be a \textit{rational semigroup} or a \textit{transcendental semigroup} depending on whether $\mathscr{F}$ is a collection of rational functions or transcendental entire functions. In particular, $ S $ is said to be a \textit{polynomial semigroup} if $\mathscr{F}$ is a collection of polynomial functions. The transcendental semigroups or polynomial semigroups are also called \textit{entire semigroups}.  The holomorphic semigroup $S$ is \textit{abelian} if 
  $$
 f_\alpha \circ f_\beta  =f_\beta \circ f_\alpha 
 $$  
for all generators $f_{\alpha}$, $f_{\beta}$  of $ S $. 
Indeed, the semigroup $S =\langle f_{\alpha} \rangle $ is the intersection of all semigroups containing $\mathscr{F} $ and it is a least semigroup containing $\mathscr{F}$. 
Note that each $f \in S$ is constructed through the composition of finite number of functions $f_{\alpha}$ for $\alpha \in \Delta $. That is, 
$$
f =f_{\alpha_1}\circ f_{\alpha_2}\circ \cdots\circ f_{\alpha_m}
$$ 
for some $ m \in \mathbb{N} $ and $ \alpha_{1}, \alpha _{2}, \ldots \alpha_{m} \in \{1,2, \ldots, m\}$. 
A semigroup generated by finitely many holomorphic functions $f_{j}, (j = 1, 2, \ldots,  n) $  is called a \textit{finitely generated  holomorphic semigroup} and we write 
$$S= \langle f_{1},f_{2},\ldots,f_{n} \rangle.
$$
If $S$ is generated by only one holomorphic function $f$, then $S$ is called a \textit{cyclic holomorphic semigroup},  and we write 
 $$
 S = \langle f\rangle.
 $$ In this case, each $g \in S$ can be written as $g = f^n$, where $f^n$ is the nth iterate of $f$ with itself.  We say that $S = \langle f\rangle$ is a \textit{trivial semigroup}.
 
 \begin{exam}\label{hsex}
 Let $ S $ be a set consisting all of powers $ z $ which are either all powers of 2 or all powers of 3 or product of all powers of 2 and 3. Then $ S $ forms a semigroup under functional composition. It is a finitely generated polynomial semigroup generated by two polynomials $ z \to z^{2} $  and $ z \to z^{3} $. In this case, $ S =\langle z^{2}, z^{3} \rangle $. 
 \end{exam}

 \begin{exam}\label{hsex1}
 Let $ \alpha \in \mathbb{C} $ such that \textbf{Re} $ \alpha \geq 0 $. For any $ k \in \mathbb{N} $, the function 
 $$
 f_{k}(z) = e^{-\alpha k}z \;\; \text{for all} \;\; z \in\mathbb{C}
 $$
is holomorphic in the complex plane $ \mathbb{C} $ and so $ S = \{f_{k}: k \in \mathbb{N}\}$ is a holomorphic (in particular, entire) semigroup. 
\end{exam}
The particular examples of holomorphic semigroup given in Example \ref{hsex1} are 
$$ 
S = \{e^{-k} z: k \in \mathbb{N}\}, 
$$
a dilation semigroup,  and 
$$ S = \{e^{ik} z: k \in \mathbb{N}\}, 
$$ 
a rotation semigroup. Note that semigroup of  Example \ref{hsex1} is a linear action on $ \mathbb{C} $ or certain subsets thereof. That is, each one parameter semigroup (discrete or continuous) which is a linear action has the form 
$$ 
S = \{f_{k}: k \in \mathbb{N} \}
$$
 where $ f_{k}(z) = e^{-\alpha k}z  $ for some $ \alpha \in \mathbb{C} $ and  \textbf{Re} $ \alpha \geq 0 $. 
 
 \begin{exam}\label{hsex2}
If we choose index set $ \Delta $, a set of positive rational numbers $ \mathbb{Q^{+}} $ in Example \ref{hsex1}, then each function 
$$
 f_{r}(z) = e^{-\alpha r}z \;\; \text{for all} \;\; z \in\mathbb{C}\; \text{and}\; r \in \mathbb{Q^{+}}
 $$
 for some $ \alpha \in \mathbb{C} $ such that \textbf{Re} $\alpha \geq 0 $, can be written as a finite composition of $ f_{t}(z) = e^{-\alpha t}z $ with $ t \in \mathbb{Q^{+}} $. Hence $ S = \langle f_{t} \rangle _{t \in \mathbb{Q^{+}}}$ is a holomorphic semigroup generated by the set $ \{ f_{t} : t \in \mathbb{Q^{+}}\} $.
\end{exam}
Note that if we choose  $ \Delta =  \mathbb{R^{+}} \cup \{0\}$, the set of non-negative real numbers, then the semigroup of Example \ref{hsex1} is a (semi) flow on $ \mathbb{C} $. Such a semigroup is also an example of one parameter continuous holomorphic semigroup. That is, in general, if the index set of holomorphic semigroup is a set of non-negative real numbers, then the semigroup is known as \textit{flow} or \textit{one parameter continuous semigroup}. If $ \Delta =(-R, R),\; R>0 $  (a real number), then $ S $ is a group which is called \textit{one parameter holomorphic group} and each element $ f_{t} $ in $ S $ is an biholomorphic mapping (automorphism) of $ \mathbb{C_{\infty}} $ whose inverse is $ f_{-t} $.  

\begin{exam}
Define a function ${f_t}, \;  t \in \mathbb{R}$, by 
$$
f_t(z) =  \dfrac{z + \tan ht}{z\tan ht + 1 } .
$$
It is easy to see that for each $ t\in \mathbb{R}, \; f_{t} $ is an automorphism,  Hence, one parameter holomorphic group $S =\{f_t : t \in \mathbb{R}\}$ is a flow of automorphisms. 
\end{exam}

\section{Ideal Theory of holomorphic semigroups}

Unless otherwise stated, we may assume onward that $ S $ is a discrete holomorphic semigroup. There are certain subsets of semigroups with a stronger closure property rather than that of subsemigroups. 

\begin{dfn}[\textbf{Left ideal, right ideal and two sided ideal}]
Let $ I $ be a non-empty subset of  holomorphic semigroup $ S $.  We say $ I  $ is a left ideal (or right ideal) of $ S $ if $ f \circ h \in I $ (or $ h \circ f \in I $) for all $ f \in S $ and $ h \in I $, that is, $ SI \subset I $ (or $ IS \subset I) $. We say $ I $ is two sided ideal (or simply ideal) if it both left and right ideal. 
\end{dfn}
Note that if $ S $ is an abelian semigroup, then the notions of left ideal, right ideal, and two sided ideal coincide. 
For any non-empty  subset $ K $ of a holomorphic semigroup $ S $,  the sets 
$$
SK =\{f\circ g: f \in S, g \in K \} = \bigcup_{g \in K}S \circ g,
$$
$$
KS =\{g\circ f: f \in S, g \in K \} = \bigcup_{g \in K}g \circ S,
$$
and 
$$
SKS =\{f\circ g \circ h: f, h \in S, g \in K \} = \bigcup_{g \in K}S \circ g \circ S
$$
are respectively left, right and two sided ideals. Likewise, for any $ g \in S $, the sets $ S \circ g $, $ g \circ S $ and $ S \circ g \circ S $ are respectively left, right and two sided ideals.
In general,  $ g $ may not be in $ g \circ S $ ( or $S\circ g  $ or $ S \circ g \circ S $) for each $ g \in S $.  If it happens to be in $ g \circ S $ ( or $S\circ g  $ or $ S \circ g \circ S $), then $ g = g \circ f $  (or $ g = f \circ g $ or $ g = f \circ g \circ h$) for some $ f, h \in S $. In this case, $ g \circ S $ (or $S \circ g $ or $S\circ g \circ S  $) is a smallest right (or left or two sided) ideal containing $ g $, which is a right (or left or two sided) ideal generated by $ g $.  Otherwise, $ g \circ S $ (or $g \circ S  $ or $S\circ g \circ S  $) is said to be quasi-generated by $ g $. It is obvious that the union of any non-empty family of left (or right or two sided) ideals of  $ S $ is a left (or right or two sided) ideal of $ S $. On the basis of some topological structure of the complement of $ g \circ S $ in $ S $, we can define the following types of holomorphic semigroup. Recall that a space $ X $ is compact if every open covering of $ X $ contains a finite subcover. 
\begin{dfn}[\textbf{F-semigroup and C-semigroup}]\label{2.3}
Let $ S $ is a holomorphic semigroup and $ g \in S $. Then we say
\begin{enumerate}
\item $ S $ is a  F-(right) semigroup if $ S -  g \circ S   $ is finite;
\item $ S $ is a C-(right) semigroup if $ S -  g \circ S   $ is relatively compact (that is, $ \overline{S -  g \circ S}   $ is compact in $ S $).
\end{enumerate}
\end{dfn}
Analogously, we can define F-(left) semigroup  and C-(left) semigroup of any holomorphic semigroup $ S $. We say only F-semigroup and C-semigroup onward for such holomrphic semigroup $ S $ on the assumption that left/right is clear from the context.  For example, holomorphic semigroup of Example \ref{hsex1} is both F-semigroup and C-semigroup and that of Example \ref{hsex2} is a C-semigroup.  The holomorphic trivial semigroup is also both F-semigroup and C-semigroup. Every (semi) flow, that is, one parameter continuous semigroup is a C-semigroup.  

There are certain type of left (or right) ideals which are connected to two sided ideals. That is, on the basis of such ideals of holomorphic semigroups, we can construct two sided ideals. This ideal structure is defined as follows.
\begin{dfn}[\textbf{Minimal left (or right) ideal}]
A left (or right) ideal $ M $ of holomorphic semigroup $ S $ is minimal if for every left (or right) ideal $ I $ of $ S $ such that $ I \subseteq M $, we have $ M =I $. 
\end{dfn}
Note that minimal left (or right) ideal of $ S $ may be empty. If it is non-empty for a certain holomorphic semigroup $ S $, then for every $ f \in M $, there must be $ M \circ f= M $ (or $ f \circ M= M $)  and $ S \circ f= M $ (or $ f \circ S= M $). Also, a minimal left (or right) ideal is always contained in every two sided ideal of $ S $.  We can also make two sided ideal by the help of minimal left (or right) ideals. For any holomorphic semigroup $ S $, let us define
$$ 
K(S) = \bigcup \{M : M \; \text{is a minimal left (or right) ideal of}\; S\}
$$
Since $ K(S) $ is non-empty if and only if $ S $ has at least one minimal left (or right) ideal and in such a case it is itself a minimal left (or right) ideal. So, as stated above,  it is contained in every two sided ideal of $ S $. Since for all $ f \in S $ 
$$ 
K(S)\circ f = \bigcup \{M\circ f : M \; \text{is a minimal left (or right) ideal of}\; S\} \subseteq K(S)
$$
So, $ K(S) $ is  also left (or right) ideal as well. From this discussion, we can conclude the following result. 

\begin{prop}[\textbf{Example of minimal two sided ideal}]
For any holomorphic semigroup $ S, \; K(S) $ is a minimal two sided ideal of $ S $ if it is non-empty. 
\end{prop}
\begin{proof}
See for instance in {\cite[Theorem 2.9] {berg2}}. 
\end{proof}

We can define a special type of holomorphic semigroup $ S $ where $ K(S) $ is non-empty. This type semigroup has some special features such as every left (or right) ideal of $ S $ include minimal one and every left (or right) ideal of $ S $ contains a special element which is called an idempotent.  Recall that an element $ e \in S $ is an \textit{idempotent} if $ e \circ e =e $. 
\begin{dfn}[\textbf{Abundant semigroup}]
A holomorphic semigroup $ S $ is abundant if every left (or right) ideal of $ S $ includes a minimal one and every minimal left (or right) ideal contains an idempotent element. 
\end{dfn}
It is obvious that $ K(S) \neq \emptyset $ if $ S $ is abundant.
There is a topologically significant nice examples of abundant semigroups which can be defined as follows.
\begin{dfn}[\textbf{Compact right topological holomorphic semigroup}]\label{2.4}
Let $ S $ be a holomorphic semigroup and $ g\in S $. 
\begin{enumerate}
\item We define a right translation map  $ F_{g}: S \to S$ with respect to $ g $  by $F_{g}(h) =h \circ g  $ for all $ h \in S $. 
\item We define a compact holomorphic right topological semigroup  by the pair $ (S, \tau) $, where $ \tau  $ is a topology on $ S $ such that the space $ (S, \tau) $ is compact and Hausdorff and right translation map $ F_{g} $ for every $ g \in S $ is continuous with respect to $ \tau $. 
\end{enumerate}
\end{dfn}
Note that a left translation  map and a compact holomorphic left topological semigroup are defined similarly. Also note that in a compact right topological semigroup, we do not require that left translation maps are continuous. We say only translation map and compact holomorphic topological semigroup if left or right is clear from the context.  

\begin{exam} \label{hsex3}
Let $ S = \{ f_{t}:  t \in \mathbb{Q^{+}} \cup \{0\} \} $ be a holomrphic semigroup, where $ f_{t} $ is a function of Example \ref{hsex2} for all $ t \in \mathbb{Q^{+}} \cup \{0\} $. The collection of all  subsets of $ S $ forms a topology $ \tau $ on $ S $ and hence it is also open cover of $ S $. There are some finite number of elements in $ \tau $ that can cover $ S $. Hence space $ (X, \tau) $ is compact and Hausdorff as well. Therefore,  this semigroup is compact holomrphic topological semigroup and hence abundant.  
\end{exam}
Note that in the holomorphic semigroup $ S $ of  Example \ref{hsex3}, there is an element $ f_{0}(z) = e^{-\alpha 0}z$ such that $ f_{0}\circ f_{0} = f_{0} $ which an idempotent by definition.  Therefore, unless holomrphic semigroup in general, abundant semigroup and in particular,  compact holomorphic topological semigroup do have idempotents. 

\begin{theorem}[\textbf{Idempotents exist for holomorphic semigroup}]  \label{th2.1}
Let $ S $ be a compact holomorphic topological semigroup. Then there is an element $e \in  S $ such that $ e \circ e = e $. 
\end{theorem}
This theorem can be proved as a standard application of Zorn's lemma from set theory. It states that \textit{if every chain $ \mathscr{C} $ in a partially ordered set $ (S, \leq ) $ has upper bound in $ S $, then $ (S, \leq ) $ has a maximal element}. Note that partial ordered set is a system consisting of non-empty set $ S $ and a relation denoted by $ \leq $ satisfying the properties of anti-symmetry, reflexivity and transitivity.  A chain $ \mathscr{C} $ in a partial ordered set $ (S, \leq ) $ is subset of $ S $ such that for every $ x, y \in  \mathscr{C}$, either $ x \leq y $ or $ x \geq y $. An element $ m \in S $ is a maximal element of $ (S, \leq ) $ if $ m \leq x $ for $ x \in S $ implies $ m =x $. 

\begin{proof}[Sketch of the Proof of Theorem \ref{th2.1}]
The proof of this theorem follows via the following two facts:\\
\textit{Fact 1}: $ S $ has a minimal close subsemigroup. \\
Let $ \tau $ be a family of all closed subsemigroups  of semigroup $ S $. Then $ \tau \neq \emptyset $ and it is a topology of closed sets partially ordered by the reverse inclusion. Let $ \mathscr{C} \subset \tau$ be a chain in $ (\tau, \supseteq) $. Since $ S $ is compact and $ \mathscr{C} $ has the finite intersection property. So, $ \cap_{T \in \mathscr{C}} T$ is non-empty and serves as a least upper bound of $ \mathscr{C} $. Then by Zorn's lemma, $ \tau $ has a maximal element $ M $ (say) where $M \subseteq \cap_{T \in \mathscr{C}} T $. In reality of this context, $ M $ is minimal closed subsemigroup of $ S $.  \\
\textit{Fact 2}: If $ e \in M $, then $ M =\{e\}$ and $ e $ is an idempotent. \\
We can consider two cases of the proof of this fact 2.\\
\textit{Case 1}: We prove $ M= M \circ e =\{e\}  $. Let $ e \in M $. Then $ M \circ e $ is a subsemigroup of $ M $. The map $F_{e}: M \to M \circ e,\; F_{e} (h)\to h \circ e $ is a right translation map of topological holomorphic semigroup $ S $ restricted to $ M $. Then it is continuous. Since $ M $ is compact so $ M \circ e $ is also compact as a image of compact set under continuous map $F_{e}  $. Since by fact 1, $ M $ is minimal and $ M \circ e $ is non-empty, we must $ M= M \circ e =\{e\}  $. This proves that $ e $ is an idempotent. \\
\textit{Case 2}: We prove $ N =\{f \in M: f \circ e = e\} = M$.  By the construction of set $ N $, it is subset of $ M $ and closed under functional composition. This shows that $ N $ is a subsemigroup of $ M $. By the case 1, $ e \in M =M \circ e $ can be written as $ e = f \circ e $ for some $ f \in M $. This shows that $ e \in N \neq \emptyset $. Finally, $ N $ can be written as the intersection of closed subsets of $ M $ and $ F_{e}^{-1}\{\{e\}\} $. This proved $ N $ is a non-empty closed subsemigroup of $ M $. Then as in the case 1, $ N =\{e\} $. 
\end{proof}

Since every left (or right) ideal is a subsemigroup of $ S $. So, from the proof of this theorem \ref{th2.1}, we can say the following important facts for every compact holomorphic topological semigroup $ S $. 
\begin{enumerate}
\item Every left (or right) ideal of the form $ S \circ f $ (or $ f \circ S $) is the image of continuous right translation map $ F_{f}(g) =g \circ f $ for all $ g \in S $. So $ S \circ f $ (or $ f \circ S $) is closed in $ S $. 
\item For every left (or right) ideal $ I $ of $ S $, we have $ S \circ f \subset I $ (or $ f \circ S \subset I) $. Thus every left (or right) ideal contains closed ideal. 
\item For every minimal left (or right) ideal $ M $ of $ S $, we have $ M =S\circ f $ (or $ M =f\circ S $ for all $ f \in M $. This shows that every minimal left (or right) ideal of $ S $ is closed. 
\item Every left (or right) ideal contains an idempotent element.
\item Let $ I \subset J $, where $ J $ is an arbitrary  left (or right) ideal and $ I$ is a closed left (or right) ideal of $ S $. By Zorn's lemma, there is a minimal element $ L $
of $ (\mathcal{I}, \subseteq) $, where $ \mathcal{I}$ is the family of closed left (or right) ideals contained in $ I $. This shows $ L $ is a minimal left (or right) ideal of $ S $. That is, every left  (or right) ideal $ S $ contains minimal left (or right) ideal. 
\end{enumerate}
From all of above facts, we may conclude the following result as well. 
\begin{theorem}[\textbf{An example of abundant semigroup}]
Every compact holomorphic topological semigroup is abundant.
\end{theorem}

There is a close connection between minimal left ideals and certain idempotents (minimal idempotents). If we suppose a set of idempotents of arbitrary holomorphic semigroup by $ E(S) $, then it may be empty but it is non-empty for abundant holomorphic semigroup. Minimal idempotent can be defined in a partial order relation on $ E(S) $. 

\begin{dfn}[\textbf{Dominated element of set of idempotents}]
Let $ E(S) $ be a set of idempotents of (abundant) holomorphic semigroup $ S $.  We say $ e $ is domonated by $ f $ on  $ E(S) $ and write
$$
e \leq f\;\; \text{if and only if}\;\; e \circ f =f \circ e =e\;\; \text{for all}\;\;  e, f \in E(S). 
$$
\end{dfn}
This is clearly partial order relation on $ E(S) $. That is, for all $ e, f, g \in E(S) $, we have
\begin{enumerate}
\item $ e \leq e $.
\item $ e \leq f $ and $ f \leq e  \Longrightarrow e =f$. 
\item $ e \leq f $ and $ f \leq g  \Longrightarrow e\leq g$. 
\end{enumerate} 
An element $ e \in E(S) $ is \textit{minimal} if there is no element of $ E(S) $ strictly less than $ e $. That is, if $ h \leq  e $, then $ h =e $ for all $ h \in E(S) $. In a abundant holomorphic semigroup, minimal idempotents are tightly connected to minimal left (or right) ideals. 
\begin{theorem} [\textbf{A connection between minimal ideal and minimal idempotant}]\label{2.4}
Let $ S $ be a holomorphic abundant semigroup and $ e \in E(S) $. Then 
\begin{enumerate}
\item If $ L \subseteq S\circ e $ (or $ L \subseteq e\circ S $) is a minimal left (or right) ideal, then there is some idempotent $ f \in L $ such that $ f \leq e $.
\item $ e $ is a minimal idempotent if and only if $ e \in L $ for some minimal left (or right) ideal $ L $, that is, if and only if $ e \in K(S) $.
\item $ e $ is minimal idempotent if and only if the left (or right) ideal $ L =S \circ e $ (or $ L =e \circ S $) is  minimal.
\item There is some minimal idempotent $ f $ such that $ f \leq e $. 
\end{enumerate}
\end{theorem}
\begin{proof}
We prove this theorem for left ideal $ L $ of semigroup $ S $. By symmetry, the same is true for right ideal. 

(1). Since $ L $  is minimal left (or right) ideal, so there an idempotent in $ L$.  Since $ S $ is abundant, so there is an $ i \in L \cap E(S) $. Since $ i \in L \subseteq S\circ e $, so $ i= s\circ e $  for some $ s \in S $. Now 
$$
 i\circ e =s \circ e \circ e =s\circ e =i. 
$$
 Let $ e \circ i = f $. We will show that $ f$ is our required element. First of all $ f \in L $ because $ L $ is a left ideal. Secondly, $ f $ is an idempotent because 
 $$ 
 f\circ f = e \circ i \circ e \circ i = e \circ i \circ i =e \circ i =f .
 $$  
Finally, $ f $ is an minimal idempotent because 
 $$ 
 e \circ f =e \circ  e \circ i = e \circ i =f 
 $$ 
 and 
 $$ f \circ e =e \circ  i \circ e = e \circ i =f. 
 $$ 

(2). Let us suppose that $ e $ is a minimal idempotent of $ S $.  Let us make a set 
$$
\mathcal{L} = \{L \subseteq S: L \; \text{is a minimal left ideal of} \; S \}
$$
Since $ S $ is abundant, so there is $ L\in \mathcal{L} $ such that $ L \subseteq S \circ e $. By (1), there is $ f \in L \cap E(S) $ such that $ f \leq e $. Since $ e$ is a minimal idempotent, we must have $ f =e $. Therefore, $ e \in L $. Conversely assume that $ e \in L $ for some minimal left ideal $ L \in \mathcal{L}$. We have to show that $ e $ is an minimal idempotent. Let $ h \in E(S) $ with $ h \leq e $. Then $ S\circ h \subseteq S \circ e $. By minimality of $ L $, we can write 
$$
 L = S\circ h = S \circ e. 
 $$
So $ e = g \circ h $ for some $ g \in S $. By $ h \leq e $, we can write $ h= e \circ h $ and so 
$$
 h =e \circ h = g\circ h \circ h= g \circ h =e. 
 $$ 
This proves $ e $ is a minimal idempotent. 

(3). This follows from the observation of converse part of (b) that if $ e \in L $ and $ L $ is minimal left ideal, then $ L =S\circ e $. 

(4). By abundancy of $ S $  and (a), there is a minimal left ideal such that $ L \subseteq S \circ e $ and $ f \in L \cap E(S) $ such that $ f \leq e $. Then $ f $ is minimal by (b). 
\end{proof}

From the proof of this theorem, it is not hard to conclude that minimal idempotents of abundant holomorphic semigroup $ S $ are just those in $ K(S) $, that is, they are from $ E(S) \cap K(S) $. This Theorem \ref{2.4} is also an important source of alternative definition of minimal idempotent. That is, on the basis of this Theorem \ref{2.4} (2), minimal idempotent can also be defined as follows. \textit{An idempotent $ e $ on semigroup $ S $ is minimal if it belongs to the minimal left (or right) ideal}. Also, from the same Theorem \ref{2.4} (3), we can say that minimal left (or right) ideal always is of the form $ S \circ e $ (or $ e \circ S $), where $ e $ is a minimal idempotent in $ S $. 

\section{On filters and ultrafilters of holomorphic semigroups}
Now, we are interested to investigate more rigorous examples of compact topological holomorphic  semigroups (that is,  abundant semigroups).  First,  we define the following notion of filters and ultrafilters:

\begin{dfn}[\textbf{Filter and ultrafilter of a holomorphic semigroup}]
Let $ S $ be a holomrphic semigroup.  A \textit{filter} $ \mathscr{F} $ on $ S $ is a family of non-empty subsets of $ S $ with the following properties: 
\begin{enumerate}
\item $ \emptyset \notin \mathscr{F} $. 
\item $ A \in \mathscr{F} $ and $ A \subset B  \Longrightarrow B \in \mathscr{F}$.
\item $ A, B \in \mathscr{F}\Longrightarrow A \cap B \in \mathscr{F} $. 
\end{enumerate}
A filter $ \mathscr{F} $ is an \textit{ultrafilter} (or \emph{maximal filter}, or \emph{prime filter}) if there is $ n \in \mathbb{N} $ 
such that 
$$
S = A_{1} \cup A_{2} \cup \ldots \cup A_{n}
$$
for some $ i $ in $ 1 \leq i \leq n $ and $ A_{i} \in \mathscr{F}$. 
The semigroup $ S $ together with the (ultra) filter $ \mathscr{F} $ is called a semigroup (ultra) filtered by $ \mathscr{F} $ or an (\emph{ultra}) \emph{filtered semigroup}. 
\end{dfn}
The condition stated for ultrafilter implies that there is no filter on $ S $ which is strictly finer than $ \mathscr{F} $. In other words, $ \mathscr{F} $ is an ultrafilter if and only if for every two disjoint subsets $ A, B \in S $ such that $ A \cup B \in \mathscr{F} $, then either $ A \in \mathscr{F} $ or $ B \in \mathscr{F} $. Equivalently, if $ \mathscr{F} $ ultrafilter and $ E \subset S $, then either $ E \in \mathscr{F} $ or $ S -E \in \mathscr{F} $. 
From the conditions (1) and (3), we can say that an (ultra) filter $ \mathscr{F} $ satisfies the \textit{finite intersection property}, that is, \textit{intersection of any finite number of sets of the family is non-empty}.  From this facts, we can say that every filter is extendible to an ultrafilter.  Not only filters but any family of subsets of $ S $ can be extendibe to ultrafilters if and only if this family has finite intersection proporty. Also, every non-empty subsets of $ S $ is contained in an ultrafilter. The space of all ultrafilters on a semigroup $ S $ is denoted by $ \beta S $ which is a  \textit{Stone-Cech compactification} of $ S $. The space $ \beta S $ is quite large in the sense that its cardinality is equal to the cardinality of $ \mathscr{P}(\mathscr{P}(S)) $ where $ \mathscr{P}(S) $ represets power set of $ S $ and indeed, it is extremely large to be metrizable of $ \beta S $. Filter and ultrafilter are topological terms and most powerful tools in topology and set theory, but we have adopted these terms from \cite{berg1} where it is   defined as  an extension set of the additive  semigroup $ \mathbb{N} $.  

Note that  each ultrafilter $ \mathscr{F} \in \beta S $ can be identified with finitely additive $ \{0, 1\} $-valued probability measure $ \mu_{\mathscr{F}} $ on the power set $ \mathscr{P} (S) $ by
$$
\mu_{\mathscr{F}}(A) = 1
$$
if and only if $A \in \mathscr{F}$. 
It is obvious that 
$$
\mu_{\mathscr{F}}(\emptyset) = 0\;\;\; \text{and} \;\;\; \mu_{\mathscr{F}}(S) = 1.  
$$ 
With this notion, every $ f \in S $ is naturally identified with an ultrafilter by defining 
\begin{equation*}
\mu_{f}(A) =  
\begin{cases}
1 & \text{if}\;\;\; f \in A, \\
0 & \text{if}\;\;\;  f \notin A
\end{cases}
\end{equation*}
This is nothing other than the set $ \{A \subset S : f \in A\}$. Such type ultrafilters  $ \mu_{f}, \; f \in S $ are called \textit{fixed} or \textit{principal}. That is, semigroup $ S $ itself can be identified as a subset of all principal untrafilters from $ \beta S $.  Ultrafilters from the subset of $ \beta S - S $ are called \textit{free} or \textit{non-principal}. Note that, by assumption, $ S $ is an infinite semigroup, there are ultrafilters on $ S $ which don't have the form $ \mu_{f}, \; f \in S $. In other words, the canonical embedding $ f \to  \mu_{f}, \; f \in S $ is not a bijection.  That is, for such a semigroup, an ultrafilter $ \mathscr{F} $ is free if and only if every set in  $ \mathscr{F} $ is infinite. Note that a filter $\mathscr{F}  $ on $ S $ is free if the intersection of all sets in $ \mathscr{F} $ is empty, otherwise, $ \mathscr{F} $ is fixed by any element of the intersection of sets of $ \mathscr{F} $. We can define cluster point of a filter $ \mathscr{F} $ and on the basis of this term,  we can get equivalent principal and free ultrafilter. If point $ l $ is in all sets of a filter $ \mathscr{F} $,  then it is known as \textit{cluster point}.  If cluster point is in ultrafilter, then it is set of all set containing that point, which is nothing other than principal ultrafilter. Clearly, an ultrafilter can have at most one cluster point. Clearly, an ultrafilter wih no cluster point is free. For every $ f, g \in S $, we have 
$$
\mu_{f} \circ \mu_{g}  =\mu_{f \circ g}.
$$ 
This shows that principal ultrafilters forms a semigroup which is isomorphic to $ S $. 

\begin{exam}
Let $ E \neq \emptyset$ be a subset of a holomorphic semigroup $ S $. The set
$$
\mathscr{F} = \{F \subseteq S: E \subseteq F \} =\langle E \rangle 
$$
is a principal filter on $ S $ generated by $ E $. It is indeed, the least filter on $  S $ containing $ E $. If $ E =\{f\} $ for some $ f \in S $, then $ \mathscr{F} $ is an ultrafilter on $ S $. 
\end{exam}

\begin{exam}
Let $ S $ be a holomorphic semigroup. Define 
$$
\mathscr{F} = \{E\subset S : S - E \; \text{is finite}\}. 
$$
It is easy to verify that $ \mathscr{F} $ is generated by the family 
$$
\{S-\{f\}: f \in S\}
$$ 
forms a filter, the Frechet filter on $ S $. It is not an ultrafilter. 
\end{exam}

Next, we make $ \beta S $ a topological space which is, in fact,  compact and Hausdorff as shown in the following theorem. 
\begin{theorem} [\textbf{$ \beta S $} as a compact Housdorff space]\label{2.2}
Let $ S $ be a holomorphic semigroup.  The space $ \beta S $ of all ultrafilters of $ S $ is a topological space. In fact, it is a compact Hausdorff space. 
\end{theorem}
\begin{proof}
For given a set $ A \subset S $, define 
$$ \overline{A} = \{\mathscr{F} \in \beta S: A \in \mathscr{F} \}\;\;\;\;\; \text{and} \;\;\;\;\;  \mathscr{B} =\{\overline{A} : A \subseteq S\}.  
 $$
Then, for any $ A, B \subset S $,  the following  follows easily.
 $$
 \overline{A\cup B} = \overline{A} \cup \overline{B} \;\;\;  \overline{A\cap B} = \overline{A} \cap \overline{B} \;\;\; \text{and}\;\;\; \overline{S-A} = \beta S - \overline{A}. 
 $$ 
 Since $ S = \bigcup A $ for all $ A \subset S $. So 
 $$ 
 \overline{S} =\overline{\bigcup_{A\subset S} A} = \bigcup_{\overline{A} \in \mathscr{B}} \overline{A}  =\beta S.  
 $$ 
 This proves that $ \mathscr{B} $ is a basis for the topology on $ \beta S $ where topology $ \tau $ is defined by  the collection of all unions of elements of  the basis $ \mathscr{B} $.  Let $T \in \tau  $.  Then,  there are some  $ \overline{A} \in \mathscr{B}  $ such that $ T = \bigcup{\overline{A}} $ where $ \overline{A} $ is an open set. We can consider $ \tau $ an open cover of $ \beta S $. Each element $ T $ of the topology $ \tau $ is covered by finite open sets and so $ \beta S \in \tau $ is also cover by finite open sets $ \overline{A} \in\mathscr{B} $. This proves $ \beta S $ is a compact space.
 
 To make $ \beta S $ a Hausdorff space,  let us suppose a pair of distinct ultrafilters $ \mathscr{E}, \mathscr{F} \in \beta S  $.  Since by above construction, every element of basis $ \mathscr{B} $ is both open and closed in $ \beta S $. So, there exist open sets  $ T_{1} $ and $ T_{2} $   containing $ \mathscr{E}$ and  $\mathscr{F} $ respectively where  $ T_{1} = \bigcup{\overline{A_{i}}} $ for some $ i\leq n $ and $ T_{2} = \bigcup{\overline{A_{j}}} $ for some $ j \leq n $ with $ i \neq j $  such that $ T_{1}\cap T_{2}  =\emptyset$. 
\end{proof}
 
In the proof of Theorem \ref{2.2}, we considered the set  $\mathscr{B} =\{\overline{A} : A \subset S\}$  where $\overline{A} = \{\mathscr{F} \in \beta S: A \in \mathscr{F} \}$ as a basis of the topology $ \tau $ on $ \beta S $.  In this case, we say that $ \overline{A}$ is the \textit{Stone set} corresponding to $ A\subseteq S $, $ \mathscr{B} $ is the \textit{Stone base} or simply, \textit{(ultra)filter base} of $ \beta S $ and $ \tau $ is the \textit{Stone topology} and the space $ \beta S $ is the \emph{Stone-Cech compactification} of semigroup $ S $.   Two (ultra) filter bases  are said to be equivalent if they generate the same filter. Any family of sets satisfying the finite intersection property is a subbasis of for a (ultra) filter $ \mathscr{F} $ since the family together with the finite intersection of its members is a (ultra) filter base.  
 
The space $ \beta S $ with basis $ \mathscr{B} $ is a bonafide example of compact Hausdorff space which is isomorphic to $ \overline{S} = \{\mathscr{F} \in \beta S: S \in \mathscr{F} \}$. So, it seems natural to extend the operation of functional composition from semigroup $ S $ to compact Hausdorff  space $ \beta S $  as follows. 
In this sense,  $ \beta S $ can be  a good candidate for compact holomorphic topological semigroup.  According to Definition \ref{2.4}, to make $ \beta S $ a compact topological holomorphic semigroup,  we need to define binary operation and extension map as a translation map on $ \beta S $. 

\begin{dfn}[\textbf{Operation on $ \beta S $}]\label{2.6}
Let $ \beta S $ be the space of ultrafilters of holomrphic semigroup $ S $. For any $ A \subset S $ and $ f \in S $. Define a set $ A \circ f =\{g \in S: g \circ f \in A \}$.  For any $ \mathscr{E}, \mathscr{F} \in \beta S $, define convolution by
$$
\mathscr{E} \star \mathscr{F} =\{A \subseteq S: \{f \in S: A\circ f \in \mathscr{E} \} \in \mathscr{F} \} 
$$
\end{dfn}
The operation $ \star $ on $\beta S  $ is well defined and associative but not commutative. In fact, we prove the following result.

\begin{theorem}[\textbf{$ \beta S $ is a semigroup}]\label{2.13}
For any $ \mathscr{E}, \mathscr{F} \in \beta S $, $ \mathscr{E} \star \mathscr{F}  $ is an  ultrafilter  and $ \star $ on $\beta S  $ is associative. 
\end{theorem}
\begin{proof}
From Definition \ref{2.6}, it is clear that $ \emptyset \notin \mathscr{E} \star \mathscr{F} $. Let $ A, B \in \mathscr{E} \star \mathscr{F} $. Then again by the definition \ref{2.6}, we have
$$
\{f \in S: A\circ f \in \mathscr{E} \} \in \mathscr{F}\;\; \text{and}\;\;\{f \in S: B\circ f \in \mathscr{E} \} \in \mathscr{F}
$$
 To show $ A \cap B \in  \mathscr{E} \star \mathscr{F}$, we need to show 
 $$
 \{f \in S: (A\cap B)\circ f \in \mathscr{E} \} \in \mathscr{F}. 
$$ 
But this happens easily as 
 $$
 \{f \in S: (A\cap B)\circ f \in \mathscr{E} \} = \{f \in S: A\circ f \in \mathscr{E} \} \cap \{f \in S: B\circ f \in \mathscr{E} \}  \in \mathscr{F}. 
 $$
To show $\mathscr{E} \star \mathscr{F} \neq \emptyset$, assume $ A \subset S $ but  $ A\notin \mathscr{E} \star \mathscr{F} $.  Then $\{f \in S: B\circ f \in \mathscr{E} \}  \notin \mathscr{F}   $. This is equivalent with $\{f \in S: A\circ f \in \mathscr{E} \}^{c}  \in \mathscr{F}$. This is true precisely when $\{f \in S: A^{c}\circ f \in \mathscr{E} \}  \in \mathscr{F} $. This shows that $ A^{c}\in \mathscr{E} \star \mathscr{F}  $.

Finally, we check the associativity of the operation  $ * $ on $ \beta S $. Let $ A \subset S $ and $ \mathscr{E}, \mathscr{F}, \mathscr{G} \in \beta S $. Then
\begin{eqnarray*}
A \in  \mathscr{E} \star (\mathscr{F}\star \mathscr{G}) 
&  \iff  &  \{f \in S: A\circ f \in \mathscr{E} \} \in \mathscr{F}\circ \mathscr{G}\\
& \iff & \{g \in S : (\{f \in S: A\circ f \in \mathscr{E} \} \circ g) \in \mathscr{F}\} \in \mathscr{G}\\
& \iff &  \{g \in S : \{f \in S: A\circ g\circ f \in \mathscr{E} \}  \in \mathscr{F}\} \in \mathscr{G}\\
& \iff & \{g \in S : A\circ g \in \mathscr{E} \star \mathscr{F} \} \in \mathscr{G}\\
& \iff  &  A \in  (\mathscr{E} \star \mathscr{F})\star \mathscr{G} \\
\end{eqnarray*}
\end{proof}

As we stated before, if ultrafilters $ \mu_{f}$ for $f \in S $ are principal, then the convolution operation $ * $ on $ \beta S $ correspond to the operation $ \circ $ of holomorphic semigroup $ S $.  
Note that the convolution defined in  Definition \ref{2.6} is the unique extension of the operation $ \circ $ in $ S $. 
This leads to the translation map of  Definition \ref{2.4} if we define it as $ F_{\mathscr{F}} (\mathscr{E}) = \mathscr{E} \star \mathscr{F}$. Indeed, this map is continuous as shown below.
\begin{theorem} [\textbf{The convolution operation as a continuous self map}]
For a fixed $ \mathscr{E} \in \beta S $, the function $ F_{\mathscr{F}} (\mathscr{E}) = \mathscr{E} \star \mathscr{F}$ is continuous self map of $ \beta S $ for any $ \mathscr{F} \in \beta S $.
\end{theorem}
\begin{proof}
For $ \mathscr{F} \in \beta S $, let $ \mathcal{U} $ be a neighborhood of $ F_{\mathscr{F}} (\mathscr{E}) $. The result can be proved if there exists a neighborhood $ \overline{\mathcal{V}}  $ of $ \mathscr{E} $ such that $ F_{\mathscr{F}} (\mathscr{G}) \in \mathcal{U}$ for any $ \mathscr{G} \in \overline{\mathcal{V}}$ where $\overline{\mathcal{V}}=\{\mathscr{H} \in \beta S: \mathcal{V}  \in \mathscr{H} \} $. Let $ A \subset S $ such that $ F_{\mathscr{F}} (\mathscr{E}) = \mathscr{E} \star \mathscr{F} \in \overline{A} \subset \mathcal{U}$. Then $A \in \mathscr{E} \star \mathscr{F}  $. Define a set $ \mathcal{V} =\{f \in S: A \circ f \in \mathscr{E}\} $. By  Definition \ref{2.6} of $ \mathscr{E} \star \mathscr{F}  $, $ \mathcal{V} \in \mathscr{F} $ or $ \mathscr{F} \in \overline{\mathcal{V}} $. If $ \mathscr{G} \in \overline{\mathcal{V}}  $, then $ \mathcal{V} =\{f \in S: A \circ f \in \mathscr{F}\} \in \mathscr{G} $. This implies that $ A \in \mathscr{F} \star \mathscr{G} = F_{\mathscr{F}}(\mathscr{G}) $ or $ F_{\mathscr{F}}(\mathscr{G}) \in \overline{A} \in \mathcal{U} $. 
\end{proof}

Now, we are able to get very rigorous example of compact holomorphic topological semigroup. This is a space of ultrafilters $ \beta S $ with operation of convolution of Definition \ref{2.6}. Indeed, this space $ \beta S $ is a compact holomorphic topological semigroup on the basis of Theorems \ref{2.2} and \ref{2.13}. It is considered a nice characteristic property of the space $ \beta S $.     
As we mentioned in  Theorem \ref{th2.1}, such semigroups are known to have idempotents.  An idempotent (in this case, it is called \textit{idempotent ultrafilter}) in $ \beta S $ is an element $ \mathscr{I} $ such that $ \mathscr{I} \star \mathscr{I} =\mathscr{I} $.  In such a case, we have $ A \in \mathscr{I} \iff A \in  \mathscr{I} \star \mathscr{I} \iff \{f \in S: A \circ f \in \mathscr{I}\} \in \mathscr{I}$. Also if $ \mathscr{I}  $ is an idempotent ultrafilter, then $ f \circ S \in \mathscr{I} $ for all $ f \in S $. Note that an idempotent ultrafilter can not be principal ultrafilter.  
There is an investigation concerning the number of idempotent ultrafilters contained in $ \beta S $. In fact, $ \beta S $ has $ 2^{c} $ idempotants, where $ c $ is the power of continuum. This follows from the fact that $ \beta S $ has $ 2^{c} $ disjoint closed subsemigroups (see for instance \cite{dou}).

From the proof of Theorem \ref{th2.1} and little bit discussion of above paragraph,  we noticed that a compact holomorphic topological semigroup $ S $ has minimal close subsemigroup $ M $. Now, we investigate something more on minimal close subsemigroup of $ \beta S $. 

\begin{theorem}[\textbf{Mnimal closed subsemigroups are left(or right) ideals}]\label{2.8}
The minimal close subsemigroups of $ \beta S $ are precisely minimal left (or right) ideals.  
\end{theorem}
\begin{proof}
Suppose  that $ \mathcal{I} $ is a minimal close subsemigroup of $ \beta S $. 
Then by fact 2 in the proof of the theorem \ref{th2.1}, we can write $ \mathcal{I} =\{\mathscr{I} \} $, where $\mathscr{I}  $ is an idempotent ultrafilter of $ \beta S $. Which is a minimal (minimum as well) left (or right) ideal of $ \beta S $. 
\end{proof}

This Theorem \ref{2.8} is also a converse of the fact that every left (or right) ideal is a subsemigroup of the semigroup $ S $. Note that in general such type of converse may not hold but according to Theorem \ref{2.8}, if subsemigroups are closed and minimal, they are precisely minimal left (or right) ideals. Also, if semigroup is abundant, such type of close minimal subsemigroups exist and they are minimal left or right ideals.

\end{document}